\documentclass[10pt,a4paper]{article}
\usepackage[utf8]{inputenc}
\usepackage{amsmath}
\usepackage{amsfonts}
\usepackage{amssymb}
\usepackage{geometry}
\usepackage{graphicx}
\usepackage{amsthm}
\usepackage{dsfont}
\usepackage{color}

\newtheorem{thm}{Theorem}[section]
\newtheorem{prop}{Proposition}[section]
\newtheorem{cor}{Corollary}[section]
\newtheorem{rmk}{Remark}[section]
\newtheorem{lma}{Lemma}[section]

\newcommand{\E}{\mathbb{E}}

\newcommand \Esp {\mathbb{E}}

\def\N{{\rm I\kern-0.16em N}}
\def\R{{\rm I\kern-0.16em R}}
\def\E{{\rm I\kern-0.16em E}}
\def\P{{\rm I\kern-0.16em P}}
\def\F{{\rm I\kern-0.16em F}}
\def\B{{\rm I\kern-0.16em B}}
\def\C{{\rm I\kern-0.46em C}}
\def\G{{\rm I\kern-0.50em G}}

\author{Thibault Pautrel}

\title{New asymptotics for the mean number of zeros of random trigonometric polynomials with strongly dependent Gaussian coefficients} 

\begin{document}
\maketitle

\abstract{
We consider random trigonometric polynomials of the form
\[
f_n(t):=\frac{1}{\sqrt{n}} \sum_{k=1}^{n}a_k \cos(k t)+b_k \sin(k t),
\]
where $(a_k)_{k\geq 1}$ and $(b_k)_{k\geq 1}$ are two independent stationary Gaussian processes with the same correlation function $\rho: k \mapsto \cos(k\alpha)$, with $\alpha\geq 0$. We show that the asymptotics of the expected number of real zeros differ from the universal one $\frac{2}{\sqrt{3}}$, holding in the case of independent or weakly dependent coefficients. More precisely, for all $\varepsilon>0$, for all $\ell \in (\sqrt{2},2]$, there exists $\alpha \geq 0$ and $n\geq 1$ large enough such that
\[
\left|\frac{\Esp\left[\mathcal{N}(f_n,[0,2\pi])\right]}{n}-\ell\right|\leq \varepsilon,
\]
where $\mathcal N(f_n,[0,2\pi])$ denotes the number of real zeros of the function $f_n$ in the interval $[0,2\pi]$. Therefore, this result provides the first example where the expected number of real zeros do not converge as $n$ goes to infinity by exhibiting a whole range of possible limits ranging from $\sqrt{2}$ to 2. 
}

\tableofcontents

\section{Introduction and statement of the results}

\subsection{Real zeros of random trigonometric polynomials}
There is tremendous amount of literature about complex or real zeros of random polynomials and their asymptotics as the degree of the latter goes to infinity. Recently, the universality of these asymptotics has been established in a certain number of models, see e.g. \cite{Kac43,IM68,Far86, Mat10,Muk18,NNV14,DNV18} in the case of algebraic polynomials and \cite{AP15,ADL,Fla17, IKM16, ADP19} in the case of trigonometric polynomials. The notion of universality stands here for the fact that these asymptotics do not depend on the choice of the law of the random entries, and to a certain extent, nor their correlation. 

\par
\bigskip
For example, in the case of trigonometric polynomials, it was shown in the last references that the first order asymptotics of the expected number of real zeros is indeed universal under very mild assumptions on the random coefficients, e.g. even in the presence of an arbitrary long-range correlation. This naturally raises the question of the existence of choices of ``exotic'' random entries such that the asymptotics of the expected number of real zeros do not coincide with the universal one.
\par
\bigskip

We address this question here by exhibiting, for the first time, a simple model of random trigonometric polynomials, whose average number of real zeros does not converge as their degree goes to infinity. Our model belongs to the large class random trigonometric polynomials of the form
\[
f_n(t):=\frac{1}{\sqrt{n}} \sum_{k=1}^{n}a_k \cos(kt)+b_k \sin(k t), \quad t \in \mathbb R,
\]
where $(a_k)_{k\geq 1}$ and $(b_k)_{k\geq 1}$ are two independent stationary Gaussian processes with correlation function $\rho: \N \to \R$, namely $\Esp[a_ka_l]=\Esp[b_k b_l]=:\rho(|k-l|)$ and $\Esp[a_kb_l]=0$ for all $k, l \geq 1$. 
Thanks to Bochner's theorem, we then know that $\rho$ is given by the Fourier transform of a measure $\mu$, called the spectral measure. The case where $\rho(k)=0$ for all $k \geq 1$ of course corresponds to independent Gaussian coefficients as first studied by Dunnage in \cite{Dun66}. Latter, in \cite{Sam78} and \cite{RS84}, the authors considered the two ``extreme'' cases where 
$\Esp[a_ia_j]=\rho_0 \in ]0,1[$ and $\Esp[a_ia_j]=\rho_0^{|i-j|}$ respectively. More recently, the authors of \cite{ADP19} considered the case where the spectral measure admits a density satisfying mild hypotheses. In all these cases, it was shown that $\mathcal N(f_n,[0,2\pi])$,   the number of real zeros of the random function $f_n$ in the interval $[0,2\pi]$, obeys the same limit
\[
\lim_{n \to +\infty}\frac{\Esp[\mathcal{N}(f_n,[0,2\pi])]}{n}=\frac{2}{\sqrt{3}}.
\]
In fact, considering standard Gaussian coefficients, one way to obtain asymptotics that do not match the universal one $2/\sqrt{3}$ is to consider either palindromic entries as in \cite{FL12} or very special pairwise block entries such as in Theorem 2.3 and 2.4 of \cite{Pir19}.   We consider here the natural and purely singular case where the spectral measure is given by $\mu:=\frac{1}{2}\left(\delta_{\alpha}+\delta_{-\alpha}\right)$, for some real $\alpha\geq 0$. In other words, we consider the cosine correlation function
\[
\rho(k)=\int_{\R}e^{i k\xi} \mu(\xi) = \cos(k\alpha).
\]
If $\alpha \in \pi \mathbb Q$, the correlation function is thus periodic and the corresponding random coefficients of $f_n$ are strongly correlated at arbitrary large distance. 
If $\alpha \notin \pi \mathbb Q$, the sequence $(\rho(k))_{k \geq 0}$ is dense in $[-1,1]$ and the correlations between the random coefficients of $f_n$ becomes really intricate.
We shall see that the asymptotics of the number of real zeros of $f_n$ then heavily depends on the arithmetic nature of $\alpha$ and more precisely on the distance of $n\alpha$ to $\pi \mathbb Z$. 

\newpage
\subsection{Statement of our results} \label{section.state.resu}
Naturally, since $f_n$ is a random trigonometric polynomial of degree $n$, its number of zeros in bounded by $2n$. In the case where $n \alpha \in \pi \mathbb Z$, we show that the expected number of real zeros is maximal in the following sense.

\begin{prop}\label{prop.danspiZ}
If $\alpha=0$, then for all $n \geq 1$ we have almost surely
\[
\mathcal{N}(f_n,[0,2\pi]) = 2n.
\]
If $\alpha \in \pi \mathbb Q$ then 
\begin{equation} \label{eq.perio.deux}
\lim_{n \to +\infty} \left|\frac{\mathbb E\left[\mathcal{N}(f_n,[0,2\pi])\right]}{n} -2 \right| \mathds{1}_{n\alpha \in \pi \mathbb Z} =0.
\end{equation}
\end{prop}

The case $n \alpha \notin \pi \mathbb Z$ is more intriguing: properly renormalized, the expected number of real zeros of $f_n$ does not converge as $n$ goes to infinity and admits in fact a whole continuum of possible limits. To be more precise, let us introduce the 
 function 
$\ell^\alpha : (0,\pi) \to \mathbb R^+$ defined by 
\[
\ell^\alpha(x):=\frac{1}{4\pi^2} \int_{[0,2\pi]^2}  \sqrt{1+|g_{n\alpha}^\alpha(s,u)|^2} ds du,
\]
where 
\[
g_x^\alpha(s,u) := \frac{\sin\left( x \right) \sin\left( \frac{s-\alpha}{2}\right)\sin\left( \frac{s+\alpha}{2}\right)}{\sin^2\left( \frac{u-x}{2}\right)\sin^2\left(\frac{s+\alpha}{2}\right)  +\sin^2\left( \frac{u+x}{2}\right)\sin^2\left(\frac{s-\alpha}{2}\right)}.
\]
In Section \ref{sec.getell} below, we examine the properties of $\ell^{\alpha}$ and its pointwise limit as $\alpha$ goes to zero
\[
\displaystyle \ell^0:x\mapsto \frac{1}{2\pi}\int_{0}^{2\pi}\sqrt{1+\frac{\sin^2(x)}{(1-\cos(u)\cos(x))^2}}du.
\]
\begin{figure}[ht]  
\hspace{0.45cm} \begin{minipage}[b]{0.45\linewidth}
   \centering
   \begin{center}\includegraphics[scale=0.22]{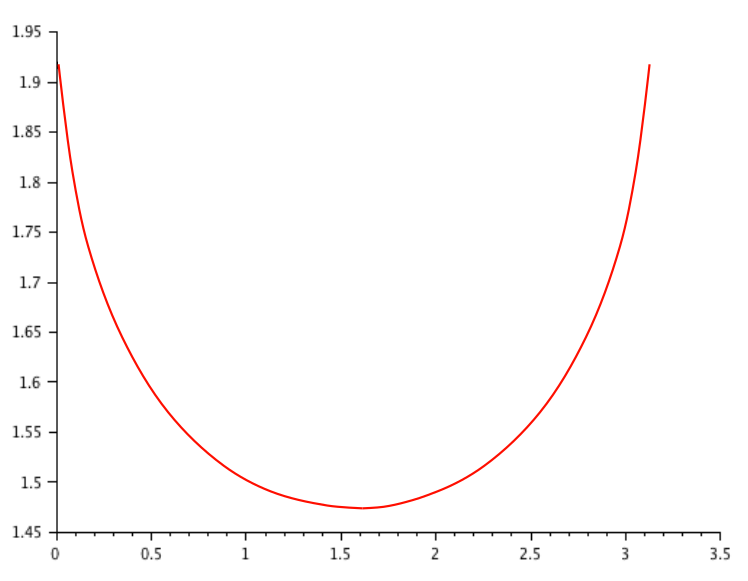} \caption{Graph of $\ell^{1/2}(x)$}\end{center}
  \end{minipage}
  \begin{minipage}[b]{0.45\linewidth}
   \centering
   \begin{center}\includegraphics[scale=0.22]{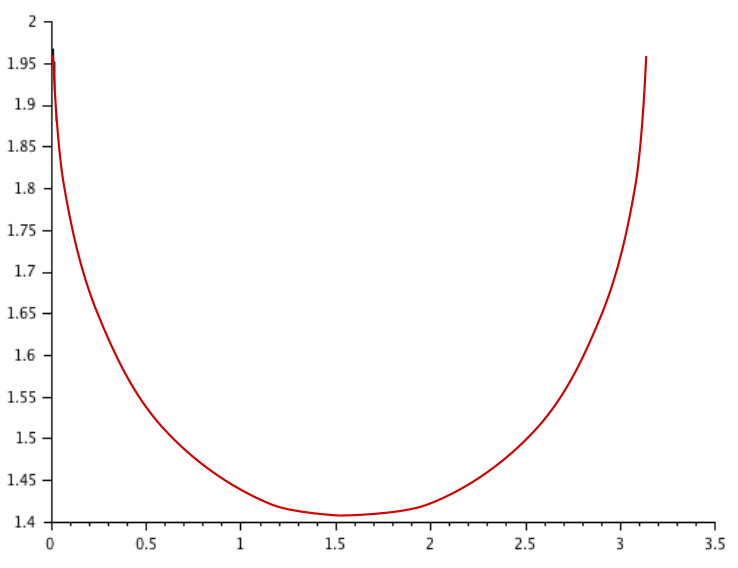} \caption{Graph of $\ell(x)$}\end{center}
  \end{minipage}

\end{figure}\par
\noindent
The main result of the paper is then the following one.
\begin{thm}\label{theo.main}
For all $0<\beta<1$ and for all $n$ large enough such that $n\alpha \notin \pi \mathbb Z$, we have 
\[
 \left|\frac{\mathbb E\left[\mathcal{N}(f_n,[0,2\pi])\right]}{n} -\ell^a( n\alpha \,\mathrm{mod} \,\pi) \right| =O\left(  \frac{1}{n^{\beta}(1-|\cos(n\alpha)|)^2}\right) +o(1).
\]
\end{thm}
The above theorem shows that if $n$ is sufficiently large but $n\alpha$ stays away enough from $\pi \mathbb Z$, then the expected number of real zeros on $f_n$ divided by $n$ is close to the value of the function $\ell^\alpha$ at the point $n\alpha \,\mathrm{mod} \, \pi$.
In particular, if $\alpha \in \pi \mathbb Q$, then the sequence $(n\alpha \, \text{mod} \, \pi)_{n \geq 1}$ takes values in a finite set $S$. From the above Theorem \ref{theo.main}, we can then deduce the following corollary.

\begin{cor} If $\alpha \in \pi \mathbb Q$, then for all $x \in S \backslash \{0\}$
\[
\lim_{n \to +\infty}  \left|\frac{\mathbb E\left[\mathcal{N}(f_n,[0,2\pi])\right]}{n} -\ell^\alpha(x) \right| \mathds{1}_{n\alpha = x \,\mathrm{mod}\, \pi}=0.
\]
In particular $n^{-1} \mathbb E\left[\mathcal{N}(f_n,[0,2\pi])\right]$ does not converge as $n$ goes to infinity.\label{cor.perio}
\end{cor}

Now if $\alpha \notin \pi \mathbb Q$, the sequence $(n\alpha \, \text{mod} \, \pi)_{n \geq 1}$ is dense in $[0, \pi]$ and from Theorem \ref{theo.main}, one then deduces that $n^{-1} \mathbb E\left[\mathcal{N}(f_n,[0,2\pi])\right]$ admits a whole continuum of possible limits. 
\begin{cor}Let us fix $x \in (0, \pi)$ and consider a increasing subsequence $(\varphi(n))_{n \geq 1}$ such that $\varphi(n)\alpha$ converges to $x$ as $n$ goes to infinity. Then
\[
\lim_{n \to +\infty}  \left|\frac{\mathbb E\left[\mathcal{N}(f_{\varphi(n)},[0,2\pi])\right]}{\varphi(n)} -\ell^\alpha(x) \right|=0.
\]\label{cor.pasperio}
\end{cor}

\begin{cor} \label{cor.resume}
For all $\varepsilon >0$, for all $\ell \in (\sqrt{2},2]$, there exists $\alpha=\alpha(\ell) \geq 0$ small enough and infinitely many integers $n$ such that 
$$\left|\frac{\Esp\left[\mathcal{N}(f_{n},[0,2\pi])\right]}{n}-\ell\right| \leq \varepsilon$$ 
where the Gaussian entries $(a_k)_{k\geq 1}$ and $(b_k)_{k \geq 1}$ of $f_{n}$ admit $\frac{\delta_{\alpha}+\delta_{-\alpha}}{2}$ as spectral measure. \end{cor}

\begin{rmk}
For sake of clarity, we only deal here with a spectral measure $\mu$ with one atom $\alpha$ and its opposite $-\alpha$, but the method employed will work for any finite combination of atoms 
$( \pm \alpha_i)_i$.
\end{rmk}

The rest of the paper is devoted to the proofs of the results stated above. Namely, in the next Section \ref{sec.piZ}, we give the proof of Proposition \ref{prop.danspiZ}, starting from the very simple case $\alpha=0$ and then generalizing to the case where $\alpha \in \pi \mathbb Q$ and $n\alpha \in \pi \mathbb Z$. The last Section \ref{sec.main} is devoted to the proof of the main Theorem \ref{theo.main} and its corollaries in the case where $n\alpha \notin \pi \mathbb Z$. 
In this case, the study of the number of zeros is split into to parts: in Section \ref{sec.away} we determine the number of zeros away from the atoms $\pm \alpha$ of the spectral measure $\mu$. Finally, the numbers of zeros in the neighborhood of the atoms is shown to be negligible in the last Section \ref{sec.near}.

\section{Asymptotics in the case $n \alpha \in \pi \mathbb Z$ } \label{sec.piZ}
In this Section, we give the proof of Proposition \ref{prop.danspiZ} describing the asymptotics of the number of real zeros of $f_n$ under the condition $n \alpha \in \pi \mathbb Z$. 
\subsection{The case $\alpha=0$}
Let us first consider the very particular case where  $\alpha=0$ i.e. the correlation function $\rho$ is constant equal to one. 
\begin{prop}\label{pro.perio.easy}
Suppose that $\alpha=0$, i.e. $\rho(k)=1$ for all $k \in \mathbb N$, then almost surely, for all $n \geq 1$ we have 
\[
\mathcal{N}(f_{n},[0,2\pi]) =2n.
\]
\end{prop}

\begin{proof}
Under the condition $\alpha=0$, the function $f_n$ has the simple form 
\[
f_n(t)=\frac{1}{\sqrt{n}}\left( A\sum_{k=1}^{n}\cos(kt)+B\sum_{k=1}^{n}\sin(kt)\right),
\]
where $A,B$ are two independent standard Gaussian variables.
For a.e. $x \in [0,2\pi]$, standard trigonometric calculations give 
\[
\sum_{k=1}^{n}\cos(kx) = \cos\left (\frac{n+1}{2}x\right) \frac{\sin(nx/2)}{\sin(x/2)} \;\; \text{and}\;\; \sum_{k=1}^{n}\sin(kx) =\sin\left (\frac{n+1}{2}x\right) \frac{\sin(nx/2)}{\sin(x/2)},
\]
so that
\[
f_n(t)=0 \iff \left (A \cos\left (\frac{n+1}{2}t\right) +B \sin \left (\frac{n+1}{2}t \right) \right) \frac{\sin (nt/2)}{\sin(t/2)}=0.
\]
We have thus $n-1$ deterministic zeros corresponding to 
\[
\sin(nt/2)=0 \iff t \in \left \{\frac{2\pi}{n}, \dots, \frac{2(n-1)\pi}{n} \right \},
\]
and $n+1$ random zeros given by   
\begin{eqnarray*}
\tan \left ( \frac{n+1}{2} t \right) = -\frac{A}{B} \sim \mathrm{Cauchy} 
&\iff& t(\omega) = \frac{2\pi}{n+1}U(\omega) + \frac{2k \pi}{n+1}\, , \, k \in \{0, \dots, n\}
\end{eqnarray*}
where $U=\pi/2 -\frac{1}{\pi} \arctan(-A/B)$ is uniform on $[0,1]$.
\end{proof}

\subsection{The case $\alpha \in \pi \mathbb Q$ and $n\alpha \in \pi \mathbb Z$}

Let us now suppose that $\alpha=\frac{2\pi p}{q}$ for positive and coprime integers $p$ and $q$, i.e. the correlation sequence $(\rho(k))_{k}$ is $q-$periodic. In this case, if $n=qr$ for some positive integer $r$, we have $n\alpha \in \mathbb Z$ and $f_n$ admits the following factorization
\[
f_{n}(t)=\frac{1}{\sqrt{n}} \sum_{k=1}^{q} \left (a_k \sum_{\ell=0}^{r-1}\cos((\ell q +k)t) + b_k\sum_{\ell=0}^{r-1} \sin((\ell q+k)t) \right) = \frac{1}{\sqrt{n}} \; \widetilde{f}_{n}(t) \times \frac{\sin\left (\frac{nt}{2}\right)}{\sin\left (\frac{qt}{2}\right)},
\]
where we have set
\[
\widetilde{f}_{n}(t):=\sum_{k=1}^{q} a_k \cos\left (kt+\frac{(n-q)t}{2}\right) + b_k \sin\left (kt+\frac{(n-q)t}{2}\right).
\] 
The above factorization of $f_n$ invites to distinguish deterministic and random zeros. We have $n-q$ deterministic zeros given by
\[
\sin \left ( \frac{nt}{2}\right) = 0 \;\; \text{and} \;\; \sin\left (\frac{qt}{2}\right) \neq 0 \iff t \in \left \lbrace \frac{2 k \pi}{n}, k\in \{0, \dots, n-1\}, r \nmid k \right\rbrace.
\]
Therefore the second statement in Proposition \ref{prop.danspiZ} follows from the following result which implies that, in the above framework, the expected number of real zeros of $\widetilde{f}_{n}$ is asymptotic to $n$.
\begin{prop} \label{pro.perio}
As $n$ tends to infinity, we have 
\[
\liminf_{\substack{n \to +\infty\\ q \mid n}}  \frac{1}{n} \mathbb E\left[\mathcal N(\widetilde{f}_{n}, [0, 2\pi])\right]  \geq 1.
\]
\end{prop}
\begin{proof}
A direct computation shows that if $q \mid n$
\[
\begin{array}{ll}
\mathbb E\left[ \widetilde{f}_{n}(t)^2\right] & =\displaystyle{ \sum_{k,\ell=1}^q \rho(k-\ell) \cos((k-\ell)t) = \sum_{k,\ell=1}^q  \cos((k-\ell)\alpha) \cos((k-\ell)t)} \\
\\
& =  \displaystyle{\frac{1}{2} \left[ \frac{\sin^2\left(\frac{q(\alpha+t)}{2} \right)}{\sin^2\left(\frac{(\alpha+t)}{2} \right)} + \frac{\sin^2\left(\frac{q(\alpha-t)}{2} \right)}{\sin^2\left(\frac{(\alpha-t)}{2} \right)} \right].}
\end{array}
\]
Since $q\alpha\in \pi \mathbb Z$, we have thus for $t \in [0,2\pi]$
\[
\mathbb E[ \widetilde{f}_{n}(t)^2]=0 \Longrightarrow qt/2 \in \pi \mathbb Z  \Longrightarrow t \in S_q:=\left \lbrace\frac{2\pi k}{q}, \; 0 \leq k \leq q-1\right \rbrace.
\]
For $\varepsilon>0$, set $S_q^{\varepsilon}:=\{ t \in [0, 2\pi], \mathrm{dist}(t,S_q)>\varepsilon\}$. On $S_q^{\varepsilon}$, we have $\mathbb E[ \widetilde{f}_{n}(t)^2]>0$ and applying Kac-Rice formula (see e.g. Theorem 3.2 p. 71 of \cite{AW}), we get
\begin{equation}\label{eq.kac1}
\Esp[\mathcal{N}(\widetilde{f}_{n},S_q^{\varepsilon})] = \frac{1}{\pi} \int_{S_q^{\varepsilon}}\sqrt{\frac{\mathbb E[ \widetilde{f}_{n}'(t)^2]}{\mathbb E[ \widetilde{f}_{n}(t)^2]}- \left(\frac{\mathbb E[ \widetilde{f}_{n}(t)\widetilde{f}_{n}'(t)]}{\mathbb E[ \widetilde{f}_{n}(t)^2]}\right)^2} dt.
\end{equation}
A straightforward computation shows that as $n$ goes to infinity, uniformly in $t \in S_q^{\varepsilon}$
\[
\begin{array}{ll}
\mathbb E[ \widetilde{f}_{n}'(t)^2] & =\displaystyle{\sum_{k,\ell=1}^q \rho(k-\ell)\left( k+\frac{n-q}{2} \right) \left( \ell+\frac{n-q}{2} \right) \cos((k-\ell)t)}\\
\\
& = \displaystyle{\left( \frac{n-q}{2} \right)^2\mathbb E[ \widetilde{f}_{n}'(t)^2] + o(n^2)}.
\end{array}
\]
Since $\mathbb E[ \widetilde{f}_{n}(t)^2]$ does not depend on $n$, neither does $\mathbb E[ \widetilde{f}_{n}(t)\widetilde{f}_{n}'(t)]$ so that as $n$ goes to infinity, we have uniformly in $t \in S_q^{\varepsilon}$
\[
\sqrt{\frac{\mathbb E[ \widetilde{f}_{n}'(t)^2]}{\mathbb E[ \widetilde{f}_{n}(t)^2]}- \left(\frac{\mathbb E[ \widetilde{f}_{n}(t)\widetilde{f}_{n}'(t)]}{\mathbb E[ \widetilde{f}_{n}(t)^2]}\right)^2} = \frac{n}{2} \left( 1+o(1)\right).
\]
Injecting this estimate in equation \eqref{eq.kac1}, we deduce that as $n$ goes to infinity
\[
\frac{\Esp[\mathcal{N}(\widetilde{f}_{n},S_q^{\varepsilon})]}{n} = \frac{|S_q^{\varepsilon}|}{2\pi} \left( 1+o(1)\right) = 1+ O(\varepsilon)+o(1).
\]
Letting $\varepsilon$ go to zero, we finally get that
\[
\liminf_{n \to +\infty} \frac{\Esp[\mathcal{N}(\widetilde{f}_{n},[0,2\pi])}{n} \geq \liminf_{n \to +\infty}  \frac{\Esp[\mathcal{N}(\widetilde{f}_{n},S_q^{\varepsilon})]}{n} =1.
\]
\end{proof}

\section{Asymptotics in the case $n \alpha \notin \pi \mathbb Z$ }\label{sec.main}
We now consider the more intriguing case where $n \alpha \notin \pi \mathbb Z$. Following \cite{ADP19}, the variance and covariance of $(f_n(t),f_n'(t))$ can then be written as convolutions of the spectral measure $\mu$ with explicit trigonometric kernels, namely
\begin{equation}\label{mesconv}
\mathbb E[f_n(t)^2] = K_n \ast \mu(t), \quad \mathbb E[f_n(t) f_n'(t)] = \frac{1}{2} K_n' \ast \mu(t),  \quad \mathbb E[f_n'(t)^2] = \frac{1}{\alpha_n} L_n \ast \mu(t),
\end{equation}
where $\displaystyle{K_n(x) := \frac{1}{n} \left( \frac{\sin(nx/2)}{\sin(x/2)}\right)^2}$ is the Fejer kernel, so that 
\[
K_n'(x) := \frac{2}{n} \left( \frac{\sin(nx/2)}{\sin(x/2)}\right)\left( \frac{n \cos(nx/2)}{2\sin(x/2) } -  \frac{ \sin(nx/2)\cos(x/2)}{2\sin(x/2)^2 } \right),
\]
the normalization constant $\alpha_n$ is given by $\alpha_n:=6/(n+1)(2n+1)$  and
\[
L_n(x) := \frac{\alpha_n}{n} \left| \sum_{k=0}^n k e^{ikx} \right|^2 =  \frac{\alpha_n}{n} \frac{(n+1)^2}{4 \sin(x/2)^2} \left| 1- \frac{\left(1-e^{i(n+1)x}  \right) e^{-inx} }{(n+1)\left(1- e^{ix}\right)} \right|^2.
\]

\begin{lma}For $0<\varepsilon\leq 1$, define $F_{\varepsilon}:=\{ x \in [0, 2\pi], |\sin(x/2)| \geq \varepsilon\}$. Then for all $n \geq 1 $ such that $n \varepsilon >1$, we have the uniform estimates  
\[
 \sup_{x \in F_{\varepsilon}} \left| K_n'(x)-   \frac{\sin(nx/2)\cos(nx/2)}{\sin(x/2)^2}   \right| =O\left(  \frac{1}{n \varepsilon^3}\right),
\]
\[
 \sup_{x \in F_{\varepsilon}} \left| L_n(x)- \frac{\alpha_n n }{ 4\sin(x/2)^2} \right| =O\left(  \frac{1}{n^2 \varepsilon^3} \right). 
\]\label{lem.esti}
\end{lma}
\begin{proof}
The estimate for $K_n'$ is immediate. Let us set 
\[
u:=\frac{\alpha_n}{n} \frac{(n+1)^2}{4 \sin(x/2)^2}, \quad z:= \frac{\left(1-e^{i(n+1)x}  \right) e^{-inx} }{(n+1)\left(1- e^{ix}\right)},
\]
so that $L_n(x)=u |1-z|^2 = u (1-z)(1-\overline{z})=u + u \times \left(|z|^2 - 2 \Re{(z)}\right)$. Since on $F_{\varepsilon}$ we have
\[
u \leq\frac{\alpha_n}{n} \frac{(n+1)^2}{4 \varepsilon^2}, \qquad  |z| \leq \frac{1}{(n+1) |\sin(x/2)|} \leq \frac{1}{n\varepsilon},
\]
we get that as soon as $n \varepsilon>1$ 
\[
|L_n(x)-u| \leq  |u| \times [ |z|^2 + 2 |z| ]  \leq 3 |u| \times |z|  \leq \frac{\alpha_n}{n} \frac{(n+1)^2}{4 \varepsilon^2} \times \left[ \frac{3}{n\varepsilon}\right] =O\left( \frac{1}{n^2 \varepsilon^3}\right).
\]
Moreover, we have 
\[
\left| \frac{\alpha_n}{n} \frac{(n+1)^2}{4 \sin(x/2)^2} - \frac{\alpha_n n }{ 4\sin(x/2)^2}\right| = \frac{\alpha_n}{4 \sin(x/2)^2} \left|\frac{(n+1)^2}{n} -n \right| = O\left( \frac{1}{n^2 \varepsilon^2}\right) = O\left( \frac{1}{n^2 \varepsilon^3}\right),
\]
hence the result.
\end{proof}

\noindent
In the case we consider here, the spectral measure $\mu$ is $\frac{1}{2}\left( \delta_\alpha+\delta_{-\alpha}\right)$ so that 
we have simply
\[
\begin{array}{ll}
\mathbb E[ f_n(t)^2] & = \frac{1}{2} \left( K_n(t-\alpha) + K_n(t+\alpha)\right), \\
\\
\mathbb E[ f_n(t)f_n'(t)] & = \frac{1}{4} \left( K_n'(t-\alpha) + K_n'(t+\alpha)\right),\\
\\
\mathbb E[ f_n'(t)^2] & = \frac{1}{2} \left( L_n'(t-\alpha) + L_n'(t+\alpha)\right).
\end{array}
\]
The Fejér kernel being non negative, for $n\geq 1$, we have
\[
\mathbb E[f_n(t)^2] = 0 \Rightarrow \left\{\begin{array}{c} nt \in\pi \mathbb Z \\ n\alpha \in \pi \mathbb Z. \end{array}\right.
\]
Under the assumption $n \alpha \notin \pi \mathbb Z$, the distribution of the Gaussian variable $f_n(t)$ is thus non-degenerated for all $t \in [0,2\pi]$ and as above, we can use Kac--Rice formula (see e.g. \cite{AW}) to compute the expectation of $\mathcal{N}(f_n, [0,2\pi])$, namely 
\[
\mathbb E\left[ \mathcal N(f_n,[0,2\pi])\right] = \frac{1}{\pi}\int_0^{2\pi} \sqrt{I_n(t)}dt,
\]
where 
\[
I_n(t) := \frac{1}{\alpha_n} \frac{L_n(t-\alpha)+L_n(t+\alpha)}{K_n(t-\alpha)+K_n(t+\alpha)}-\frac{1}{4}\left( \frac{K'_n(t-\alpha)+K'_n(t+\alpha)}{K_n(t-\alpha)+K_n(t+\alpha)}\right)^2.
\]
We split the computation of the integral into two parts, depending on the proximity between the integration variable $t$ and the atoms $\pm \alpha$ of the spectral measure $\mu$. 

\subsection{Away from the atoms}\label{sec.away}
Let us fix $\varepsilon >0$ and consider the set $J_{\varepsilon} := \{ t \in [0,2\pi], |\sin(\frac{t-\alpha}{2})|>\varepsilon, |\sin(\frac{t+a\alpha}{2})|>\varepsilon\}$. Thanks to Lemma \ref{lem.esti}, we have then uniformly in $t \in J_{\varepsilon}$
\[
\frac{L_n(t-\alpha)+L_n(t+\alpha)}{K_n(t-\alpha)+K_n(t+\alpha)} =  \frac{\frac{\alpha_n n^2}{4} \left(\frac{1}{\sin^2\left(\frac{t-\alpha}{2}\right) } + \frac{1}{\sin^2\left(\frac{t+\alpha}{2}\right) } \right)+O\left( \frac{1}{n\varepsilon^3}\right) }{\frac{\sin^2\left(n \frac{t-\alpha}{2}\right)}{\sin^2\left(\frac{t-\alpha}{2}\right) } +\frac{\sin^2\left(n \frac{t+\alpha}{2}\right)}{\sin^2\left(\frac{t+\alpha}{2}\right) }  }.
\]
In the same manner, we have
\[
\frac{K'_n(t-\alpha)+K'_n(t+\alpha)}{K_n(t-\alpha)+K_n(t+\alpha)} = \frac{ \frac{\sin\left(n \frac{t-\alpha}{2}\right)\cos\left(n \frac{t-\alpha}{2}\right)}{\sin^2\left(\frac{t-\alpha}{2}\right) } +\frac{\sin\left(n \frac{t+\alpha}{2}\right)\cos\left(n \frac{t+\alpha}{2}\right)}{\sin^2\left(\frac{t+\alpha}{2}\right) }+O\left( \frac{1}{n\varepsilon^3}\right)}{\frac{\sin^2\left(n \frac{t-\alpha}{2}\right)}{n \sin^2\left(\frac{t-\alpha}{2}\right) } +\frac{\sin^2\left(n \frac{t+\alpha}{2}\right)}{n \sin^2\left(\frac{t+\alpha}{2}\right) }}.   
\]
Now remark that uniformly on $J_{\varepsilon}$ we have
\[
\begin{array}{ll}
\frac{1}{\frac{\sin^2\left(n \frac{t-\alpha}{2}\right)}{ \sin^2\left(\frac{t-\alpha}{2}\right) } +\frac{\sin^2\left(n \frac{t+\alpha}{2}\right)}{ \sin^2\left(\frac{t+\alpha}{2}\right) }}  & \leq  \frac{1}{\varepsilon^2 \left( \sin^2\left(n \frac{t-\alpha}{2}\right)  +\sin^2\left(n \frac{t+\alpha}{2}\right) \right) } = \frac{1}{\varepsilon^2 (1-\cos(nt)\cos(n\alpha))} \\
&  \leq \frac{1}{\varepsilon^2 (1-|\cos(n\alpha)|)}.
\end{array}
\]
Therefore, uniformly on $J_{\varepsilon}$ we get
\[
I_n(t) = \frac{n^2}{4} \left( Q_n(t) + O\left( \frac{1}{n \varepsilon^5 (1-|\cos(n\alpha)|)}\right) \right), 
\]
where 
\[
\begin{array}{ll}
Q_n(t) & := \frac{\frac{1}{\sin^2\left(\frac{t-\alpha}{2}\right) } + \frac{1}{\sin^2\left(\frac{t+\alpha}{2}\right) } }{\frac{\sin^2\left(n \frac{t-\alpha}{2}\right)}{\sin^2\left(\frac{t-\alpha}{2}\right) } +\frac{\sin^2\left(n \frac{t+\alpha}{2}\right)}{\sin^2\left(\frac{t+\alpha}{2}\right) }  }
- \left(\frac{ \frac{\sin\left(n \frac{t-\alpha}{2}\right)\cos\left(n \frac{t-\alpha}{2}\right)}{\sin^2\left(\frac{t-\alpha}{2}\right) } +\frac{\sin\left(n \frac{t+\alpha}{2}\right)\cos\left(n \frac{t+\alpha}{2}\right)}{\sin^2\left(\frac{t+\alpha}{2}\right) }}{\frac{\sin^2\left(n \frac{t-\alpha}{2}\right)}{\sin^2\left(\frac{t-\alpha}{2}\right) } +\frac{\sin^2\left(n \frac{t+\alpha}{2}\right)}{\sin^2\left(\frac{t+\alpha}{2}\right) }}   \right)^2 \\
\\
& = 1 + \left(\frac{\sin\left( n\alpha \right) \sin\left( \frac{t-\alpha}{2}\right)\sin\left( \frac{t+\alpha}{2}\right)}{\left(\sin^2\left(n \frac{t-\alpha}{2}\right)\sin^2\left(\frac{t+\alpha}{2}\right)  +\sin^2\left(n \frac{t+\alpha}{2}\right)\sin^2\left(\frac{t-\alpha}{2}\right)  \right)} \right)^2.
\end{array}
\]
In particular, we get 
\begin{equation}\label{eq.InQn}
\frac{2}{n} \int_{J_{\varepsilon}} \sqrt{I_n(t)}dt =  \int_{J_{\varepsilon}} \sqrt{Q_n(t)}dt + O\left( \frac{1}{n \varepsilon^5 (1-|\cos(n\alpha)|)}\right).
\end{equation}
In order to make explicit the asymptotics of the right hand side of the last equation, let us now introduce an auxilary function and detail some of its properties. 
\subsubsection{An auxilary function and its properties}
\label{sec.getell}

For $x \in \mathbb R \backslash \pi \mathbb Z$, let us introduce the function $g_x^\alpha$ defined on $[0,2\pi]^2 \backslash \{\pm (\alpha,x)\}$ by 
\begin{equation}\label{eq.def.ga}
g_x^\alpha(s,u) := \frac{\sin\left( x \right) \sin\left( \frac{s-\alpha}{2}\right)\sin\left( \frac{s+\alpha}{2}\right)}{\sin^2\left( \frac{u-x}{2}\right)\sin^2\left(\frac{s+\alpha}{2}\right)  +\sin^2\left( \frac{u+x}{2}\right)\sin^2\left(\frac{s-\alpha}{2}\right)}.
\end{equation}
Remark that $u \mapsto g_x^\alpha(s,u)$ is then $2\pi-$periodic and that we have the identification
\begin{equation}\label{eq.Qngn}
Q_n(t) = 1 + |g_{n\alpha}^\alpha(t,nt)|^2.
\end{equation}
The function $(u,s) \mapsto g_x^\alpha(s,u)$ has singularities at $(s,u)=\pm (\alpha,x)$ but these sigularities are integrable in the following sense. 
\begin{lma}Let $0<\alpha<\pi$ and $0<x<\pi$. For all $0\leq \eta<1$, we have
\[
\int_{[0,2\pi]^2} |g_x^\alpha(s,u)|^{1+\eta} dsdu <+\infty.
\]\label{lem.int}
\end{lma}

\begin{proof}
Let us fix some small $\delta>0$. Outside the two Euclidean balls $B(\pm (\alpha,x),\delta)$ the function $ (s,u) \mapsto g_x^\alpha(s,u)$ is uniformly bounded hence in $\mathbb L^p$ for all $p\geq 1$, so we only need to focus on the integrability on $B(\pm (\alpha,x),\delta)$. By symmetry, we can restrict ourselves to the ball centered at $(\alpha,x)$. If we set $C:=\min(|\sin(x)| , |\sin(\alpha)|)>0$, for $\delta$ small enough we have
\[
|g_x^\alpha(s,u)| \leq \frac{4}{C} \frac{|s-\alpha|}{|s-\alpha|^2+|u-x|^2},
\] 
so that using polar coordinates $(s-\alpha,u-x)=(r \cos(\theta),r\sin(\theta))$ with $0\leq r \leq \delta$, $0\leq \theta \leq 2\pi$, we get
\[
\int_{B( (\alpha,x),\delta)} |g_x^\alpha(s,u)|^{1+\eta} dsdu \leq  \frac{8\pi}{C}  \int_0^\delta \frac{dr}{r^{\eta}} = O\left( \delta^{1-\eta}\right).
\]

\end{proof}

\begin{lma}\label{lem.regl}On any compact set $K\subset (0,\pi)$, the function $\ell^\alpha : K \to \mathbb R^+$
\[
x \mapsto   \ell^\alpha(x):= \frac{1}{4\pi^2} \int_{[0,2\pi]^2}  \sqrt{1+|g_{x}^\alpha(s,u)|^2} ds du
\]
is continuous.
\end{lma}

\begin{proof}Note that the regularity of $x \mapsto \ell^\alpha(x)$ is the same as the one of 
\[
 x \mapsto \int_{[0,2\pi]^2}  |g_{x}^\alpha(s,u)| ds du.
\]
Fix $\varepsilon>0$, from the proof of Lemma \ref{lem.int} applied with $\eta=0$, there exists $\delta>0$ small enough such that, for all $x \in K$, if  $E_x:=B( (\alpha,x),\delta) \cup B( -(\alpha,x),\delta)$ then
\[
\int_{E_x} |g_x^\alpha(s,u)| dsdu \leq  \varepsilon/4.
\]
Now, if $(s,u) \in E_x^c \cap E_{x'}^c$ the function $ x \mapsto  |g_{x}^\alpha(s,u)|$ is uniformly bounded and analytic so that choosing $\delta>0$ small enough , for $|x-x'|<\delta$ we have 
\[
\left| \int_{E_x^c \cap E_{x'}^c}   \left( |g_{x}^\alpha(s,u)| - |g_{x'}^\alpha(s,u)| \right) ds du\right| \leq \varepsilon/2.
\]
As a conclusion, we get that 
\[
\left| \int_{[0,2\pi]^2}   \left( |g_{x}^\alpha(s,u)| - |g_{x'}^\alpha(s,u)| \right) ds du\right| \leq \varepsilon.
\]

\end{proof}

The next lemma giving some properties of $g_x^{\alpha}$ which will be particularly useful in the sequel.

\begin{lma}\label{lem.regg}
\[
\sup_{\substack{s \in J_{\varepsilon} \\  u \in [0,2\pi]}} |g_x^\alpha(s,u)| =O\left(  \frac{1}{\varepsilon^2} \times \frac{1}{1-|\cos(x)|}\right),
\]
\begin{equation}\label{eqn.lip}
\sup_{\substack{s,s' \in J_{\varepsilon} \\ u \in [0, 2\pi]}} |g_x^\alpha(s,u)-g_x^\alpha(s',u) | =O\left(  \frac{|s-s'|}{\varepsilon^4|1-|\cos(x)|)^2}\right).
\end{equation}
\end{lma}

\begin{proof}
If $s\in J_{\varepsilon}$, we have uniformly in $u\in [0,2\pi]$
\[
|g_x^\alpha(s,u)| \leq \frac{1}{\varepsilon^2 \left[  \sin^2\left( \frac{u+x}{2}\right) + \sin^2\left( \frac{u-x}{2}\right)\right]}=\frac{1}{\varepsilon^2 \left( 1- \cos(u)\cos(x)\right)}\leq  \frac{1}{\varepsilon^2} \times \frac{1}{1-|\cos(x)|}.
\]
Moreover, for $s,s' \in J_{\varepsilon}$, setting $D(s):=  \left(\sin^2\left( \frac{u-x}{2}\right)\sin^2\left(\frac{s+\alpha}{2}\right)  +\sin^2\left( \frac{u+x}{2}\right)\sin^2\left(\frac{s-\alpha}{2}\right)  \right)$
\[
\begin{array}{ll}
|g_x^\alpha(s,u)-g_x^\alpha(s',u)| & \leq  \frac{| \sin\left( \frac{s-\alpha}{2}\right)\sin\left( \frac{s+\alpha}{2}\right) - \sin\left( \frac{s'-\alpha}{2}\right)\sin\left( \frac{s'+\alpha}{2}\right)|}{D(s)} +\frac{|D(s)-D(s')|}{|D(s)D(s')|}\\
\\
 & = O\left(  \frac{|s-s'|}{\varepsilon^2(1-|\cos(x)|)}\right) + O\left(  \frac{|s-s'|}{\varepsilon^4(1-|\cos(x)|)^2}\right).
\end{array}
\]
\end{proof}
\noindent
For $x \in (0,\pi)$, recall the definition of the function $\ell^0$ given in Section \ref{section.state.resu} 
\[
\ell^0(x):=\frac{1}{2\pi}\int_{0}^{2\pi} \sqrt{1+g_x^0(u)^2}du, \quad \textrm{where} \quad g_x^0(u):=\frac{\sin(x)}{1-\cos(u)\cos(x)}.
\]
The function $\ell^0$ appears naturally as the pointwise limit of $\ell^{\alpha}$ when $\alpha \in (0, \pi)$ goes to zero.

\begin{lma} \label{lem.ell.lim}
For all $x \in (0,\pi)$, we have $\displaystyle \lim_{\alpha \to 0} \ell^{\alpha}(x)= \ell^0(x)$.
\end{lma}
\begin{proof}
Let $\epsilon>0$ and let $\alpha \in \left(0,\frac{\epsilon}{2}\right)$ be small enough. We can write
\begin{equation}
\label{ell.decoup}\ell^\alpha(x)=\frac{1}{4\pi^2} \left[\int_{|s|>\epsilon} \int_{-\pi}^{\pi} \sqrt{1+g_x^{\alpha}(s,u)^2}dsdu+ \int_{|s|\leq \epsilon} \int_{-\pi}^{\pi}\sqrt{1+g_x^{\alpha}(s,u)^2}dsdu\right].
\end{equation}
For $|s|>\epsilon$, there exists a constant $C>0$ such that $\left|\sin \left(\frac{s\pm \alpha}{2}\right)\right| \geq C \epsilon$. Recalling the expression of $g_{x}^{\alpha}$ given by Equation \eqref{eq.def.ga}, we get uniformly on $u\in [-\pi,\pi]$ that 
\[
|g_x^{\alpha}(s,u)| \leq \frac{|\sin(x)|}{C\epsilon^2(1-|\cos(x)|)} \in \mathbb L^1([-\pi,\pi]^2).
\]
By dominated convergence, we obtain
\[
\lim_{\alpha \to 0} \int_{|s|>\epsilon} \int_{-\pi}^{\pi} \sqrt{1+g_x^{\alpha}(s,u)}dsdu = 2(\pi-\epsilon) \int_{-\pi}^{\pi} \sqrt{1+\frac{\sin^2(x)}{(1-\cos(u)\cos(x))^2}}ds.
\]
Let us now show that the second term in Equation (\ref{ell.decoup}) converges to zero as $\alpha$ goes to zero. 
By symmetry, we can restrict ourselves to the case $s \in [0,\epsilon]$. Recall that when $\omega$ is small enough, there exists some constants $C_1, C_2 >0$ such that $C_1|\omega| \leq |\sin(\omega)|\leq C_2|\omega|$. Thus, there exists $C>0$ such that
$$|g_x^{\alpha}(s,u)| \leq C\frac{|\sin(x)| |(s-\alpha)(s+\alpha)|}{\sin^2\left(\frac{u-x}{2}\right) (s+\alpha)^2+\sin^2\left(\frac{u+x}{2}\right) (s-\alpha)^2}.$$
Set $\delta>0$ small enough such that for all $u \in [x-\delta,x+\delta]$, we have $\left | \sin\left(\frac{u-x}{2}\right)\right| \geq C_{\delta} |u-x|$ and $\left|\sin \left(\frac{u+x}{2}\right)\right| \geq C_{\delta} \sin(x)$. Using the fact that 
$s+\alpha\geq \alpha$, we get that for some the constant $C$ which may change from line to line
\begin{eqnarray*}
\int_0^{\epsilon}\int_{x-\delta}^{x+\delta} |g_x^{a}(s,u)| duds &\leq &C\int_0^{\epsilon} \int_{x-\delta}^{x+\delta} \frac{|s^2-\alpha^2|}{(s-\alpha)^2+\alpha^2(u-x)^2}duds\\
&\leq& C\int_0^{\epsilon}\frac{|s+\alpha|}{\alpha}\arctan\left(\frac{\delta \alpha}{|s-\alpha|}\right) ds.
\end{eqnarray*}
This last integral can then be upper bounded by
\[
\begin{array}{ll}
\displaystyle{\int_0^{\epsilon}\frac{|s+\alpha|}{\alpha}\arctan\left(\frac{\delta \alpha}{|s-\alpha|}\right) ds} &\leq   \displaystyle{\int_{0}^{\epsilon}\frac{|s-\alpha|}{\alpha}\arctan\left(\frac{\delta \alpha}{|s-\alpha|}\right) ds}\\ 
\\
& \displaystyle{ +\underbrace{2\int_0^{\epsilon} \arctan \left(\frac{\delta \alpha}{|s-\alpha|} \right) ds}_{\leq C \epsilon}}.
\end{array}
\]
Now, performing the change of variable $v=\frac{s-\alpha}{\alpha}$ and using the fact that $x \mapsto x \arctan\left(\frac{1}{x}\right)$ is bounded on the real line, we get
\begin{eqnarray*}
\int_0^{\epsilon}\frac{|s-\alpha|}{\alpha}\arctan\left(\frac{\delta \alpha}{|s-\alpha|}\right) ds& \leq&  \epsilon \times \frac{\alpha}{\epsilon}\int_{-\frac{\epsilon}{\alpha}}^{\frac{\epsilon}{\alpha}}|v|\arctan\left(\frac{\delta}{|v|}\right)dv \leq C \epsilon.
\end{eqnarray*}
The same method naturally works in the neighborhood of $-x$. Otherwise, if we denote by $E_{\delta}$ the set $([x-\delta,x+\delta] \cup [-x-\delta, -x+\delta])^{c}$, there exists a constant $C_{x,\delta}$ such that for all $u$ in $E_{\delta}$, we have $\left|\sin\left(\frac{u\pm x}{2}\right)\right| \geq C_{x,\delta}$. Thus, for some constant which may again change from line to line, we get
\begin{eqnarray*}
\int_0^{\epsilon} \int_{E_{\delta}} |g_x^{\alpha}(s,u)|ds du  &\leq& C\int_0^{\epsilon} \frac{|s^2-\alpha^2|}{(s-\alpha)^2+(s+\alpha)^2}ds\\
&\leq& C\underbrace{\int_{0}^{\epsilon} \frac{|s^2+\alpha^2|}{s^2+\alpha^2}ds}_{=\epsilon}+2\alpha^2\int_0^{\epsilon}\frac{ds}{s^2+\alpha^2} \leq C(\epsilon+\alpha) \leq C\epsilon,
\end{eqnarray*} 
hence the result. 
\end{proof}
Let us conclude this section with some properties of the limit function $\ell^0(x)$.
\begin{lma} \label{lem.prop.ell}
The function $x\mapsto \ell^0(x)$ is analytic on $(0,\pi)$ and admits $x=\frac{\pi}{2}$ as a symmetry axis. Moreover, $[\sqrt{2},2) \subseteq \ell^0[(0,\pi)]$.
\end{lma}

\begin{proof}
Let $K\subset (0,\pi)$ a compact set. The function $K \times [0,2\pi] \ni (x,u) \mapsto g_x^0(u)$ is $\mathcal{C}^{\infty}$ and for all $p \geq 1$, 
\[
\left|\partial_x^{p}\sqrt{1+g_x^0(u)^2} \right| \leq \frac{C_p}{\inf_{K}(1-|\cos(x)|)^{\beta_p}} \in \mathbb L^1([0,2\pi]),
\] 
where $C_p>0$ and $\beta_p\geq 1$ are some explicit constants. Analyticity follows from dominated convergence. Using the change of variable $v=u+\pi$ and $2\pi$-periodicity of the integrand, we get that for all $z\in \left[0,\frac{\pi}{2}\right), \, \ell^0\left(z+\frac{\pi}{2}\right) =\ell^0\left(\frac{\pi}{2}-z\right)$. Therefore $x=\frac{\pi}{2}$ is a symmetry axis.
For all $x \in (0,\pi)$, since $\sin(x)\geq 0$, we have
\[
\sqrt{1+\frac{\sin^2(x)}{(1-\cos(u)\cos(x))^2}}\leq 1+\frac{\sin(x)}{1-\cos(u)\cos(x)},
\]
and the change of variable  $t=\tan(u/2)$ on $[0,\pi]$ and $[\pi,2\pi]$ yields
\[
\frac{1}{2\pi} \int_0^{2\pi} \frac{\sin(x)}{1-\cos(u)\cos(x)}du = 1.
\]
Hence $\ell(x)\leq 2$. In fact, the upper value 2 is obtained as the limit on the boundaries. \\
Set $\delta >0$ and let $x$ be small enough. We can indeed write
\[
\ell^0(x) =\frac{1}{2\pi}\int_{[-\pi,\pi]\setminus [-\delta,\delta]}\sqrt{1+g_x^0(u)^2}du +\frac{1}{2\pi} \int_{-\delta}^{\delta}\sqrt{1+g_x^0(u)^2}du.
\]
For $u \in [-\pi,\pi] \setminus [-\delta,\delta],~ \displaystyle \sqrt{1+g_x^0(u)^2}\leq 1+\frac{1}{1-\cos(\delta)},$ thus by dominated convergence, 
\begin{equation}\label{eq.small1}
\lim_{x \to 0} \frac{1}{2\pi}\int_{[-\pi,\pi]\setminus [-\delta,\delta]} \sqrt{1+\frac{\sin^2(x)}{(1-\cos(x)\cos(u))^2}}du = 1.
\end{equation}
On the other hand, for $x$ small enough, we can assume that 
\[
\int_{-\delta}^{\delta} \sqrt{1+g_x^{2}(u)}du \geq \int_{-\sqrt{x}}^{\sqrt{x}} \sqrt{1+g_x^{2}(u)}du.
\] 
Then, we get
\begin{eqnarray*}
\frac{1}{2\pi}\int_{-\sqrt{x}}^{\sqrt{x}}\sqrt{1+g_x^2(u)} du&=&\frac{1}{\pi} \int_0^{\sqrt{x}}\sqrt{1+\frac{\sin^2(x)}{(1-\cos(u)\cos(x))^2}}du\\
&\geq&\frac{\sin(x)}{\pi}\int_0^{\sqrt{x}}\frac{\sin(x)}{1-\cos(u)\cos(x)}du
\end{eqnarray*}
Since $\cos(u) \geq 1-\frac{u^2}{2}$ and $\cos(x)\geq 1-\frac{x^2}{2}$, we have $1-\cos(u)\cos(x) \leq \frac{u^2+x^2}{2}$, and thus
$$\frac{1}{2\pi}\int_{-\delta}^{\delta}\sqrt{1+g_x^2(u)}du \geq \frac{2}{\pi} \sin(x) \int_0^{\sqrt{x}}\frac{du}{u^2+x^2} =\frac{2}{\pi}\times \frac{\sin(x)}{x} \times \arctan\left(\frac{\sqrt{x}}{x}\right).$$
Hence we get 
\begin{equation}\label{eq.small2}
\displaystyle \lim_{x \to 0} \frac{1}{2\pi} \int_{[-\delta, \delta]} \sqrt{1+g_x^2(u)}du \geq 1,
\end{equation}
Finally, combining the estimates \eqref{eq.small1} and \eqref{eq.small2},  we obtain
$\lim_{x \to 0} \ell^0(x) =2$. The analogue limit as $x$ tends to $\pi$ is deduced by symmetry. Since $\ell^0 \left(\pi/2\right) =\sqrt{2}$ and $\ell^0$ is continuous, the intermediate value theorem yields that $[\sqrt{2},2) \subset \ell^0[(0,\pi)]$.
\end{proof}

\subsubsection{From Riemann sum to integral}
We can now establish the asymptotics of Equation \eqref{eq.InQn} as $n$ goes to infinity. As a first step, the integral of interest admits the following lower and upper bounds. 

\begin{lma}If $n \varepsilon>> 1$, then as $n$ goes to infinity, we have  
\[
\int_{J_{\varepsilon}} \sqrt{Q_n(t)} dt \geq \frac{1}{2\pi} \int_{[0,2\pi]^2}  \sqrt{1+|g_{n\alpha}^\alpha(s,u)|^2} \mathds{1}_{s \in J_{2\varepsilon}} ds du + O\left( \frac{1}{n \varepsilon^2 (1-|\cos(n\alpha)|)} \right) ,
\]
and
\[
\int_{J_{\varepsilon}} \sqrt{Q_n(t)} dt \leq \frac{1}{2\pi} \int_{[0,2\pi]^2}  \sqrt{1+|g_{n\alpha}^\alpha(s,u)|^2} \mathds{1}_{s \in J_{\varepsilon/2}} ds du + O\left( \frac{1}{n \varepsilon^2 (1-|\cos(n\alpha)|)} \right).
\]\label{lem.encadre}
\end{lma}
\begin{proof}
We give the proof of the upper bound, the lower bound can be treated in the exact same way. To simplify the expressions, let us set $E_{n}^k :=\left[ \frac{2\pi k}{n}, \frac{2\pi (k+1)}{n}\right]$ for $0 \leq k \leq n-1$.  We can then decompose the integral on $J_{\varepsilon}$ as
\[
\int_{J_{\varepsilon}} \sqrt{Q_n(t)} dt = \sum_{k=0}^{n-1} \int_{J_{\varepsilon} \cap E_n^k} \sqrt{Q_n(t)} dt =  \frac{1}{n} \sum_{k=0}^{n-1} \int_{0}^{2\pi} \sqrt{Q_n\left(\frac{2\pi k}{n}+\frac{u}{n} \right)} \mathds{1}_{\frac{2\pi k+u}{n}  \in  J_{\varepsilon}}  du.
\]
Now remark that if $n \varepsilon>> 1$, then for $n$ large enough, if $\frac{2\pi k+u}{n} \in J_{\varepsilon}$ we have in fact $E_n^k \subset J_{\varepsilon/2}$. Therefore 
\[
\int_{J_{\varepsilon}} \sqrt{Q_n(t)} dt \leq  \frac{1}{n} \sum_{k=0}^{n-1} \int_{0}^{2\pi} \sqrt{Q_n\left(\frac{2\pi k}{n}+\frac{u}{n} \right)} \mathds{1}_{E_n^k \subset  J_{\varepsilon/2}}  du,
\]
or equivalently using \eqref{eq.Qngn} and the $2\pi-$periodicity of $u \mapsto g_{n\alpha}^\alpha(s,u)$
\[
\int_{J_{\varepsilon}} \sqrt{Q_n(t)} dt \leq  \frac{1}{n} \sum_{k=0}^{n-1} \int_{0}^{2\pi} \sqrt{1+g_{n\alpha}^\alpha\left(\frac{2\pi k}{n}+\frac{u}{n}, u \right)} \mathds{1}_{E_n^k \subset  J_{\varepsilon/2}}  du.
\]
Using the estimate \eqref{eqn.lip} of Lemma \ref{lem.regg}, one then deduces that 
\begin{equation}\label{eq.QnRiemann}
\int_{J_{\varepsilon}} \sqrt{Q_n(t)} dt \leq  \frac{1}{n} \sum_{k=0}^{n-1} \int_{0}^{2\pi} \sqrt{1+g_{n\alpha}^\alpha\left(\frac{2\pi k}{n}, u \right)} \mathds{1}_{E_n^k \subset  J_{\varepsilon/2}}  du + O\left(  \frac{1}{n\varepsilon^4|1-|\cos(n\alpha)|)^2}\right).
\end{equation}
Using again Equation \eqref{eqn.lip} of Lemma \ref{lem.regg}, for all $0 \leq k \leq n-1$ such that $E_n^k \subset  J_{\varepsilon/2}$, we have uniformly in $u$
\[
\left| \sqrt{1+g_{n\alpha}^\alpha\left(\frac{2\pi k}{n}, u \right)} -\frac{n}{2\pi} \int_{E_n^k}  \sqrt{1+g_{n\alpha}^\alpha\left(s, u \right)} ds\right|= O\left(  \frac{1}{n\varepsilon^4|1-|\cos(n\alpha)|)^2}\right).
\]
Integrating in $u$, we thus get that for all $k$ such that $E_n^k \subset  J_{\varepsilon/2}$
\[
\begin{array}{ll}
\displaystyle{
\int_{0}^{2\pi} \sqrt{1+g_{n\alpha}^\alpha\left(\frac{2\pi k}{n}, u \right)} du}  & \displaystyle{\leq \frac{n}{2\pi} \int_0^{2\pi}\int_{E_n^k}  \sqrt{1+g_{n\alpha}^\alpha\left(s, u \right)} ds du}\\
\\
&  \displaystyle{ +\; O\left(  \frac{1}{n\varepsilon^4|1-|\cos(n\alpha)|)^2}\right)},
\end{array}
\]
and in particular
\[
\begin{array}{ll}
\displaystyle{\int_{0}^{2\pi} \sqrt{1+g_{n\alpha}^\alpha\left(\frac{2\pi k}{n}, u \right)} du\times \mathds{1}_{E_n^k \subset  J_{\varepsilon/2}}} & \displaystyle{\leq \frac{n}{2\pi} \int_0^{2\pi}\int_{E_n^k}  \sqrt{1+g_{n\alpha}^\alpha\left(s, u \right)} \mathds{1}_{s \in  J_{\varepsilon/2}}ds du } \\
\\
& \displaystyle{+ \; O\left(  \frac{1}{n\varepsilon^4|1-|\cos(n\alpha)|)^2}\right).}
\end{array}
\]
Injecting this last estimate in Equation \eqref{eq.QnRiemann} and making the sum over $0 \leq k\leq n-1$, we get 
\[
\int_{J_{\varepsilon}} \sqrt{Q_n(t)} dt \leq \frac{1}{2\pi} \int_{[0,2\pi]^2}  \sqrt{1+g_{n\alpha}^\alpha\left(s, u \right)} \mathds{1}_{s \in  J_{\varepsilon/2}} dsdu  +O\left(  \frac{1}{n\varepsilon^4|1-|\cos(n\alpha)|)^2}\right).
\]
\end{proof}
\begin{lma}Uniformly in $n$, and for all $0<\eta<1$, we have 
\[
\left|  \int_{[0,2\pi]^2}  \sqrt{1+|g_{n\alpha}^\alpha(s,u)|^2} \mathds{1}_{s \in J_{\varepsilon}} ds du -\int_{[0,2\pi]^2}  \sqrt{1+|g_{n\alpha}^\alpha(s,u)|^2} ds du\right| =O\left(\varepsilon^{\frac{\eta}{1+\eta}}\right).
\]\label{lem.intrest}
\end{lma}

\begin{proof}
Applying H\"older inequality with $p=1+\eta$ and $q=1+1/\eta$ and using Lemma \ref{lem.int}, we have 
\[
  \int_{[0,2\pi]^2}  \sqrt{1+|g_{n\alpha}^\alpha(s,u)|^2} \mathds{1}_{s \in J_{\varepsilon}^c} ds du \leq \left( \int_{[0,2\pi]^2}  \sqrt{1+|g_{n\alpha}^\alpha(s,u)|^2}^{1+\eta}ds du \right)^{\frac{1}{1+\eta}}  \times  O\left(\varepsilon^{\frac{\eta}{1+\eta}}\right).
\]
\end{proof}

Combining the estimate \eqref{eq.InQn} and Lemmas \ref{lem.encadre} and \ref{lem.intrest}, we conclude that for all $\varepsilon>0$ and $n$ large enough such that $n \varepsilon>>1$ 	then
\[
\left| \frac{4\pi}{n} \int_{J_{\varepsilon}} \sqrt{I_n(t)}dt - \int_{[0,2\pi]^2}  \sqrt{1+g_{n\alpha}^\alpha(s,u)^2} ds du \right| = O\left(\varepsilon^{\frac{\eta}{1+\eta}}\right)  +O\left(  \frac{1}{n\varepsilon^5|1-|\cos(n\alpha)|)^2}\right).
\]
\subsection{Near the atoms and conclusion} \label{sec.near}
We are left to estimate the number of real zeros of $f_n$ in the neighborhood of the atoms $\pm \alpha$ of the spectral measure $\mu$. If $\varepsilon=\varepsilon_n$ is of the form $\varepsilon_n=n^{-\beta}$ with $0<\beta<1/2$, Proposition 3.3.1 of \cite{Pir19} indeed show that 
\begin{equation}\label{eq.prox}
\frac{\mathbb E\left[\mathcal N\left(f_n, J_{\varepsilon_n}^c\right)\right]}{n}  = O\left(  \varepsilon_n\right).
\end{equation}
Therefore, we can conclude that, as soon as $\varepsilon_n$ is chosen of the form $n^{-\beta}$ for $0<\beta<1/5$, we have 
\begin{equation} \label{eq.thmmain}
\left| \frac{\mathbb E\left[\mathcal N\left(f_n,[0,2\pi]\right)\right]}{n} -\ell^\alpha(n\alpha \, \mathrm{mod}\, \pi) \right| = O\left(\varepsilon_n^{\frac{\eta}{1+\eta}}\right)  +O\left(  \frac{1}{n\varepsilon_n^5|1-|\cos(n\alpha)|)^2}\right),
\end{equation}
which finishes the proof of Theorem \ref{theo.main}. Then Corollary \ref{cor.perio} follows because uniformly in $x \in S \backslash \{0\}$, if $n\alpha\, \mathrm{mod} \, \pi = x$, then $1-|\cos(n\alpha)|=1-|\cos(x)|$ is bounded away from zero. In the last case where $\alpha \notin \pi \mathbb Q$, Corollary \ref{cor.pasperio} follows from Theorem \ref{theo.main} and the regularity of $\ell^\alpha$ established in Lemma \ref{lem.regl}.

From Lemmas \ref{lem.prop.ell}, \ref{lem.ell.lim} and the estimate (\ref{eq.thmmain}) as  $\alpha \to 0$ and $n\alpha \mod \pi \to 0$, remark that we get the same limit (\ref{eq.perio.deux}) as in Proposition \ref{prop.danspiZ}. 
In the same manner, Corollary \ref{cor.resume} follows from Corollary \ref{cor.pasperio}, Lemmas \ref{lem.ell.lim} and \ref{lem.prop.ell} for $\ell \in (\sqrt{2},2)$ and from Proposition \ref{prop.danspiZ} for $\ell=2$.


\section{Asymptotics for a mixed spectral measure}

We suppose in this section that the spectral measure $\mu$ as defined above can be written as the convex combination of a density measure and an atomic measure, i.e. 
\[
\mu =(1-\eta)\mu_{d}+\eta \frac{1}{N}\sum_{k=1}^{N}\frac{1}{2}\left(\delta_{\alpha_k}+\delta_{-\alpha_k}\right),
\]
for some $\eta \in [0,1)$, with $\alpha_k \geq 0$ for each $k=1,\dots, N$. We assume that $\mu_d$ admits a density $\varphi$ w.r.t. the Lebesgue measure on $[0,2\pi]$ which satisfies the conditions :\\
\textbf{A.1} $\varphi$ is continuous on $(0,2\pi)$~~,~~\textbf{A.2} $\inf_{t \in [0,2\pi]}\varphi(t)>0$~~,~~\textbf{A.3} $\sup_{t \in [0,2\pi]}|\varphi(t)| < +\infty$.
\begin{thm}\label{thm.dens.atom}
Under the above conditions \textbf{A.1--A.3}, as $n$ goes to infinity, we have
$$\lim_{n \to +\infty}\frac{\Esp\left[\mathcal{N}(f_n,[0,2\pi])\right]}{n}=\frac{2}{\sqrt{3}}.$$
\end{thm}

Note that the latter framework generalizes the ones of \cite{Sam78} and \cite{ADP19}. Indeed, taking $N=1, \alpha=0$ and $\varphi:=\frac{1}{2\pi} \mathds{1}_{[0,2\pi]}$ corresponds to the constant correlation function $\rho(\cdot)=\eta \in (0,1)$ of \cite{Sam78}. The case $\eta=0$ corresponds to the result obtained in \cite{ADP19}, while $\eta=1$ corresponds to the main results of this article stated in Section \ref{section.state.resu}. The last Theorem \ref{thm.dens.atom} shows that the contribution of the density part prevails over the one of the atomic part, and that the limit of $n^{-1} \Esp[\mathcal{N}(f_n,[0,2\pi])]$ is the same as the one obtained for independent or weakly-dependent stationary Gaussian processes.

For sake of clarity, we only deal here with $\mu_\alpha=\frac{1}{2}(\delta_\alpha+\delta_{-\alpha})$ for some $\alpha \geq 0$. Recall that for all compact set $K \subset [0,2\pi]$, we have then using Kac--Rice formula
\[
\Esp\left[\mathcal{N}(F_n,K)\right]=\frac{1}{\pi} \int_K \sqrt{\frac{\Esp[f_n'^2(t)]}{\Esp[f_n^2(t)]} -\frac{\Esp[f_n(t)f_n'(t)]^2}{\Esp[f_n(t)^2]^2}}dt.
\]
With the same trigonometric kernels $K_n$ and $L_n$ as defined in Section \ref{sec.main}, we have in fact
\[
\begin{array}{rl}
\displaystyle{\Esp[f_n^2(t)]} & =\displaystyle{(1-\eta)K_n \ast \varphi(t)+\frac{\eta}{2}\left(K_n(t+\alpha)+K_n(t-\alpha)\right),}
\\
\\
\displaystyle{\Esp[f_n'^2(t)]} & \displaystyle{=\frac{1-\eta}{\alpha_n} L_n \ast \varphi(t)+\frac{\eta}{2\alpha_n}\left(L_n(t+\alpha)+L_n(t-\alpha)\right),}\\
\\
\displaystyle{\Esp[f_n(t)f_n'(t)]} & \displaystyle{=\frac{1-\eta}{2}K_n' \ast \varphi(t)+\frac{\eta}{4}\left(K_n'(t+\alpha)+K_n'(t-\alpha) \right)}.
\end{array}
\]
Note that the condition $\eta \in [0,1)$ and the non-negativity of the kernels associated with assumption \textbf{A.2} ensure the non-degeneracy of $f_n$ on $[0,2\pi]$ and thus well posed nature of the last integral in Kac--Rice formula. 
The proof of Theorem \ref{thm.dens.atom} results from the combination of the three following lemmas.  First, adapting the proof of Lemma 4 of \cite{ADP19} and using the fact that $\|K_n\|_1=\|L_n\|_1=1$ for all $n\geq 1$, we get that in the neighborhood of zero (resp. $2\pi$), the mean number of zeros is negligible.
\begin{lma}
There exists a finite constant $C$ such that for $\varepsilon>0$ small enough and for all $n \geq 1$, 
\[
\frac{\Esp\left[\mathcal{N}_n(f_n,[0,\varepsilon])\right]}{n} \leq C \sqrt{\varepsilon} ~, ~~~\frac{\Esp\left[\mathcal{N}(f_n,[2\pi-\varepsilon,2\pi])\right]}{n} \leq C \sqrt{\varepsilon}.
\]
\end{lma}
\noindent
Now examining the contribution close to $\pm \alpha$, claim (\ref{eq.prox}) from Section \ref{sec.near} yields directly
\begin{lma}
For $\varepsilon=\varepsilon_n$ of the form $\varepsilon_n=n^{-\beta}$ with $0<\beta<1/2$, as $n$ goes to infinity, we have 
\[
\frac{\Esp\left[\mathcal{N}(f_n,[\pm \alpha -\varepsilon_n, \pm \alpha +\varepsilon_n])\right]}{n}= O(\varepsilon_n).
\]
\end{lma}
\noindent
Finally, we show that far enough from $0, 2\pi$ and $\pm \alpha$, we have desired contribution.
\begin{lma}
Let $I$ a compact set of $[\varepsilon_n,2\pi-\varepsilon_n]\setminus (\pm \alpha -\varepsilon_n, \pm \alpha+\varepsilon_n)$. As $n$ goes to infinity,
\[
\Esp\left[\mathcal{N}(f_n, I)\right]=\frac{n}{\pi\sqrt{3}} \text{Vol}(I) \left[1+o\left(\frac{1}{\sqrt{n} \varepsilon_n}\right)\right].
\]
\end{lma}
\begin{proof}
Using the explicit expressions of $K_n$ and $L_n$, we have, for all $t \in I$:
$$|K_n(t \pm \alpha)| \leq \frac{C}{n \varepsilon_n^2} ~~, ~~|L_n(t\pm \alpha)|\leq C \frac{1}{n \varepsilon_n^2}~ \;\text{and}~ \;\frac{1}{n}\left|K_n'(t\pm \alpha)\right| \leq C \frac{1}{n \varepsilon_n^2}.$$
The same method as in Lemma 2 of \cite{ADP19} then yields
\[
\left|\frac{1}{n}K_n' \ast \varphi(t)\right|  \leq \frac{C}{n^2\varepsilon_n^3}(\|\varphi\|_1+\sup_{t \in I}|\varphi(t)|)\leq \frac{C}{n^2 \varepsilon_n^3},
\]
since $\varphi$ is assumed to be uniformly bounded by \textbf{A.3}. Choosing $\varepsilon_n= n^{-\beta}$ with $\beta \in (0,1/2)$, the end of the proof follows the steps of the proof of Lemma 3 in \cite{ADP19}, which uses crucially the condition \textbf{A.1} to have uniform convergence of the convolutions. Thus, after normalization and the cancellation of the factor $1-\eta$ in the ratios, we get the convergence of the integrand in Kac--Rice formula to the desired universal constant, hence the result.
\end{proof}


\end{document}